\documentclass[11pt]{amsart}
\usepackage{amssymb}
\usepackage{amsmath}
\usepackage{mathrsfs}
\usepackage{amsfonts}
\usepackage{txfonts}
\usepackage{amssymb}
\usepackage{color}

\newcommand\RR{{{\mathbb R}}}

\newcommand\re{{\mathcal Re}}

\newtheorem{theo}{Theorem}[section]
\newtheorem{lemm}[theo]{Lemma}

\newtheorem{prop}[theo]{Proposition}
\newtheorem{rema}[theo]{Remark}

\begin{document}

\title[Hypoellipticity effects]
{Fractional order kinetic equations and hypoellipticity}

\author{R. Alexandre}
\address{R. Alexandre,
Department of Mathematics, Shanghai Jiao Tong University
\newline\indent
Shanghai, 200240, P. R. China
\newline\indent
and \newline\indent Irenav, Arts et Metiers Paris Tech, Ecole
Navale,
\newline\indent
Lanveoc Poulmic, Brest 29290 France }
\email{radjesvarane.alexandre@ecole-navale.fr}

\subjclass[2000]{35H10, 76P05, 84C40}

\date{February 2011}

\keywords{Boltzmann equation, hypoellipticity, non-cutoff cross-sections}

\begin{abstract}
We give simple proofs of hypoelliptic estimates for some models of kinetic equations with a fractional order diffusion part. The proofs are based on energy estimates together with F. Bouchut and B. Perthame previous ideas.
\end{abstract}

\maketitle


\section{Introduction}

Recently, after the study initiated by Morimoto and Xu \cite{morimoto-all}, the paper of Lerner and all \cite{lerner-all} was concerned with hypoelliptic effects related to a kinetic equation similar to the following one:
\begin{equation}\label{eq1}\partial_t f (t,x,v,) +v.\nabla_x f(t,x,v) + a(t,x,v) |D_v|^{2\beta} f = g
\end{equation}

Here we assume that $(t,x,v) \in \RR^{1+n+n}$ for some integer $n\geq 1$, that $g\in L^2$, where $L^2 = L^2 (\RR^{1+2n})$. We denote by $\| .\|$ the associated norm. The usual interpretation from kinetic theory is that $t$ plays the role of a time variable, $x$ the position and $v$ the velocity. The coefficient $a$ is assumed for example to be smooth and strictly positive, see below for precise hypothesis. The parameter $\beta$ is assumed to satisfy $0 < \beta \leq 1$. 

As regards Fourier transformation, we shall denote by $\tau$, $k$ and $\xi$ the Fourier variables dual to $t$, $x$ and $v$ respectively. Other notations used in \eqref{eq1} are standard, see for example \cite{hormander,stein}. Let us note immediately that the third (elliptic) term on the l.h.s. of \eqref{eq1} is not exactly similar to the one considered in \cite{lerner-all,morimoto-all} in that the behavior therein was taken as $a\ | \xi|^2$ for small frequency variables $\xi$, but this is not an important point from the point of view of $L^2$ theory. In the rest of the paper, we shall {\it always assume that all functions such as $f$, $g$ are smooth}.

For $\beta =1$, \eqref{eq1} is a well known model of Fokker Planck or Kolmogorov equation, for which one can find numerous methods for proving hypoellipticity, see for example \cite{hormander,bouchut,lerner-all,morimoto-all,perthame} and the references therein.

We refer for example to \cite{zaslavsky} for physical motivations for this type of kinetic equations. Another motivation is linked with the study of the spatially inhomogeneous Boltzmann equation without cutoff, see for example \cite{alex-review,villani,alex-solmax,al-1,al-3,amuxy-nonlinear-3,amuxy3,gressman-strain}, see also the recent results of \cite{arsenio} and references therein.

As far as we know, the study of hypoelliptic effects for problem \eqref{eq1} was initiated by Morimoto and Xu \cite{morimoto-all} and they derive therein a partial and non optimal result. This study was then completed with optimal results by Lerner and all \cite{lerner-all}, where they proved typically that $|D_x|^{\beta\over{1+2\beta}}f \in L^2$ and a similar estimate w.r.t. time variable. In both works, the authors used $L^2$ type methods.

While still working with an energy method, we want to show that a slight modification of the computations of Bouchut \cite{bouchut} can lead to the same results as Lerner and all \cite{lerner-all}, and therefore in comparison, we provide a very simple proof. One advantage is that it is very simple to keep track of the different constants depending on the given coefficients $a$, and furthermore, we avoid using any deep pseudo differential calculus. However, in order to study the model problem \eqref{eq1}, we do use one result, namely Proposition 1.1 from Bouchut \cite{bouchut}, whose proof is also elementary as it relies on averaging regularity type arguments. Bouchut's result is given by
\begin{prop}\label{prop-bouchut} [Proposition 1.1 of \cite{bouchut}] Assume $f\in L^2$, $g\in L^2$, $|D_v f|^\alpha f \in L^2$ for some $\alpha \geq 0$ and
\begin{equation}\label{transport}
\partial_t +v.\nabla_x f = g,
\end{equation}
Then
$$ \| | D_x|^{\alpha\over{1+\alpha}} f\| \lesssim \| g\|^{\alpha\over{1+\alpha}} \| | D_v|^\alpha f\|^{1\over{1+\alpha}}$$
\end{prop}

The proof done in \cite{bouchut} uses both Fourier transform w.r.t. time and space variables $(t,x)$, and arguments introduced in \cite{bouchut-desvillettes}. It might be plausible to only use Fourier transform w.r.t. variable $x$ by using the arguments of \cite{bouchut-desvillettes}. For completeness, we give yet another proof which was used in Alexandre \cite{alex-solmax} following arguments introduced by Perthame  \cite{perthame} in his study of higher moments estimates. Our proof is also elementary, but uses this time both space and velocity variables Fourier transform. However, we do not use Fourier transform w.r.t. time variable, and therefore we can also deal with the Cauchy initial value problem associated with model problem \eqref{transport}. We refer also to \cite{amuxy-nonlinear-b} for another proof involving Fourier transform w.r.t. time and space variables and a certain kind of uncertainty principle. All in all, it is now clear that any other different and simple proof of the above Proposition would be of interest.

Once given Proposition \ref{prop-bouchut}, we can proceed to study hypoelliptic effects connected with \eqref{eq1}. As usual, we shall begin to study the case of constant $a$, say $1$, that is
\begin{equation}\label{eq2}\partial_t f (t,x,v,) +v.\nabla_x f(t,x,v) +  |D_v|^{2\beta} f = g .
\end{equation}
Of course, a direct Fourier analysis is able to take care of this simple model, but recall that we are looking for energy type estimates.

Our main result is given by

\begin{theo}\label{result} 

\begin{enumerate}
\item  Let $f$ satisfy \eqref{eq2}. Then one has
$$ \| |D_v|^{2\beta} f\| +\| |D_x|^{{2\beta}\over{1+2\beta}} f \| \lesssim \| g\| .$$
\item Let $f$ satisfy \eqref{eq1}. Assume that $a =b^2\chi + a_-$ for some strictly positive constant $a_-$, a smooth positive function $b$ and a compactly supported and positive function $\chi$. Then one has
$$ \| |D_v|^{2\beta} f\| +\| |D_x|^{{2\beta}\over{1+2\beta}} f \| \lesssim C_a [\| g\|  + \| f\|].$$
\end{enumerate}

\end{theo}

As it will appear clearly in the proof, similar estimate w.r.t. space variable also holds for time variable. Furthermore, the proofs given below can also be adapted to take into account the initial value problem. Finally, the assumption on $a$ might appear strange, but this is one possible choice among many others. We mention that if $a$ is assumed to be only locally bounded from below, then all proofs adapt up to introducing cutoff functions. Finally, the proof also shows that it is not necessary to have a diffusion term as above. One might consider instead an equation such as
$$\partial_t + v.\nabla_x f + \int_ {\RR^n} {{K (t,x,v)}\over{| h|^{n+2\beta} }} [f (v+h) - f(v) ]dh =g$$
with good assumptions on the kernel $K$. This example is closer to a linear version of Boltzmann operator. Details are left to the interested reader. In any case, the main issue is concerned with the multipliers introduced in the proofs with the above kernel $K$.

\section{The free transport equation: proof of Proposition \ref{prop-bouchut}}

We are interested in the transport equation \eqref{transport}, under the assumptions of Proposition \ref{prop-bouchut}.

It is shown by Bouchut that we obtain $|D_x|^{\alpha\over{1+\alpha}} f \in L^2$. The method of proof was based on velocity arguments, see \cite{bouchut,bouchut-desvillettes} and references therein. In fact another proof is also provided in \cite{amuxy-nonlinear-b} using a kind of uncertainty principle, but which is more complex. Finally, another argument based on Perthame method \cite{perthame} is also possible, see for example Alexandre \cite{alex-solmax}. As mentioned by Bouchut, the commutator method seems to require more derivatives w.r.t. variable $v$, but in that case, the proof is very easy (see the proof in \cite{bouchut}).

We shall apply Perthame's argument for the usual Fokker Planck case below, following \cite{alex-solmax}. Note that this a Fourier method, using characteristics associated on the Fourier side which is also somehow used in the paper by Lerner and all \cite{lerner-all}.

If $\hat .$ denotes the Fourier transform with 
respect to the 
variables $(x,v)$ and $(k ,\xi  )$ the dual variables, one has
$$\partial_{t} {\hat f} - k .\nabla_{\xi } {\hat f} = \hat g .$$
Multiplying by $\bar f$, and taking the complex conjugate also, we obtain:
$$\partial_t | \hat f|^2 - k.\nabla_\xi |\hat f|^2 = 2 \re (\hat g , \bar{\hat f} )\lesssim |\hat g | |\hat f| .$$
Then
$$\mid {\hat f} (t,k ,\xi  )\mid^2 \lesssim   \int^t_{-\infty} (\mid {\hat f} {\hat 
g} \mid (k ,\xi  +sk , t-s ) ds .$$

Fix $r\geq 0$ and $D \geq 0$, and $k$. Then
$$\int^\infty_{-\infty} dt \int_\xi  |k |^r  \mid {\hat f} (t,k ,\xi  )\mid^2$$
$$ = \int^\infty_{-\infty} dt \int_{|\xi| \geq D}  |k |^r  \mid {\hat f} (t,k ,\xi  )\mid^2 +\int^\infty_{-\infty} dt \int_{|\xi |\leq D}  |k |^r  \mid {\hat f} (t,k ,\xi  )\mid^2  =A+B  . $$

For B, one has:
$$B \lesssim \int^\infty_{-\infty} dt \int_{|\xi |\leq D}  |k |^r   \int^t_{-\infty} (\mid {\hat f} {\hat 
g} \mid (k ,\xi  +sk , t-s ) ds .$$

Changing variable in $\xi$, this gives
$$B \lesssim \int^\infty_{-\infty} dt \int d\xi |k |^r   \int^\infty _{-\infty}1_{s\leq t} 1_{ | \xi -sk | \leq D}(\mid {\hat f} {\hat 
g} \mid (k ,\xi  , t-s ) ds$$

Change variables in $t$ (for fixed $s$) to get
$$B \lesssim \int^\infty_{0} dt \int d\xi |k |^r   \int^\infty _{-\infty}1_{ | \xi -sk | \leq D}(\mid {\hat f} {\hat 
g} \mid (k ,\xi  , t ) ds$$

Since $| \xi -s k| \geq || \xi| -s|k||$, it follows that
$$B \lesssim \int^\infty_{0} dt \int d\xi |k |^{r-1}D   (\mid {\hat f} {\hat 
g} \mid (k ,\xi  , t ) $$
$$ \lesssim \varepsilon \int^\infty_{0} dt \int d\xi |k |^{2(r-1)}D^2   \mid {\hat f}|^2 (k,\xi ,t)+ C_\varepsilon \int^\infty_{0} dt  \int d\xi | {\hat 
g} \mid^2 (k ,\xi  , t ) $$
for any $\varepsilon >0$.

Now for $A$, we get directly
$$A \lesssim  \int^\infty_{-\infty} dt \int d\xi |\xi\|^m D^{-m}  |k |^r  \mid {\hat f} (t,k ,\xi  )\mid^2 .$$

Choose $D = | k|^{r\over m}$. Then
$$A \lesssim  \int^\infty_{-\infty} dt \int d\xi |\xi\|^m   \mid {\hat f} (t,k ,\xi  )\mid^2 .$$

Then note that $|k |^{2(r-1)}D^2  = |k |^{2(r-1) +2{r\over m}} = |k|^r$, if we choose the value of $r$ such that $r ={{2m}\over{m+2}}$. We choose $m =2\alpha$. Therefore $r = {{2\alpha}\over{1+\alpha}}$. In conclusion with all these choices, we get, for fixed $k$, by absorbing the right hand side with the left hand side:
$$ \int^\infty_{-\infty} dt \int_\xi  |k |^{{2\alpha}\over{1+\alpha}} \mid {\hat f} (t,k ,\xi  )\mid^2 \lesssim \int^\infty_{-\infty} dt \int d\xi |\xi\|^{2\alpha}  \mid {\hat f} (t,k ,\xi  )\mid^2 + C_\varepsilon \int^\infty_{0} dt  \int d\xi | {\hat 
g} \mid^2 (k ,\xi  , t ) $$
and therefore
$$ \| |D_x|^{\alpha\over{1+\alpha}} f \| \lesssim \| |D_v|^\alpha f\| + \| g\| .$$

It should be observed that we do not have the same scaling as in Bouchut. But this can fixed easily as follows. We proceed as above, replacing $D$ by $\lambda D$ but use instead Cauchy Schwarz inequality for $B$, with the same choice of parameters:
$$B \lesssim \lambda  \int^\infty_{0} dt \int d\xi |k |^{r-1} D  (\mid {\hat f} {\hat 
g} \mid (k ,\xi  , t )  \lesssim \lambda \lbrace   \int^\infty_{0} dt \int d\xi |k |^{r}  \mid {\hat f}|^2 (k,\xi ,t) \rbrace^{1\over 2} \lbrace \int^\infty_{0} dt \int d\xi | {\hat 
g} \mid^2 (k ,\xi  , t ) \rbrace^{1\over 2}$$
while 
$$A \lesssim \lambda^{-m}  \int^\infty_{-\infty} dt \int d\xi |\xi\|^m   \mid {\hat f} (t,k ,\xi  )\mid^2 .$$

The we get an inequality such as:
$$U \lesssim \lambda U^{1\over 2} V^{1\over 2} +\lambda^{-m} W .$$

If we choose $\lambda$ such that the two terms on the r.h.s. coincide then we get, after some computations
$$U\lesssim V^{m\over{m+2}} W^{2\over{m+2}} .$$

Now integrate w.r.t. $k$ and use Holder inequality with exponent $p =m/(m+2)$ to get
$$\int U \lesssim [\int V ]^{m\over{m+2}} [\int W]^{2\over{m+2}}$$

Recalling that $m=2\alpha$, we get exactly
$$ \| D_x^{\alpha\over{1+\alpha}} f\| \lesssim \| g\|^{\alpha\over{1+\alpha}} \| |D_v|^\alpha f\|^{1\over{1+\alpha}} .$$

\begin{rema} On can also get estimations for the initial value problem, say if we consider the transport equation \eqref{transport} for say positive time and a given initial value at time $0$, $f_0$. For example, in that case, the small frequency part gives an additional term which can be estimated as follows
$$III = \int^\infty_{0} dt 
\int_{\mid\xi \mid\leq D} 
\mid k\mid^r \mid {\hat F_{0}} (k
,\xi  + tk )\mid^2  = \int^\infty_{0} dt \int_{\mid\xi  -t k\mid\leq D} 
\mid k\mid^r \mid {\hat F_{0}} ( k 
,\xi   )\mid^2 \leq$$
$$\lesssim  \int^\infty_{0} dt \int_{\RR^3_{\xi }} 1_{\mid t- 
{{\mid\xi \mid}\over{\mid k\mid}}\mid 
\leq {{D}\over{\mid k\mid}}} \mid {\hat F_{0}} (k
,\xi   )\mid^2 \lesssim \mid k\mid^{r-1}D \int_{\RR^3_{\xi }} \mid {\hat F_{0}} (k
,\xi   )\mid^2 .$$
\end{rema}

\section{Proof of the first part of Theorem \ref{result}: the constant coefficient case}

Here we shall prove the first part of Theorem \ref{result} related to $f$ satisfying \eqref{eq2}.  We shall again adapt the ideas of Bouchut, except for a modification of the test multiplier in Step 4 below.

{\bf Step 1:} We multiply \eqref{eq2} by $\bar f$ and integrate over all variables. Taking into account usual symmetry cancellation, we get
\begin{equation}\label{eq-a}
|| |D_v|^\beta f|| \lesssim || g||^{1\over 2} || f||^{1\over 2} .
\end{equation}

{\bf Step 2:} Now having in mind that we want to prove that $|D_v|^{2\beta } f\in L^2$ and knowing that $g\in L^2$, we note that
$$\partial_t f+v.\nabla_x f = G \equiv - |D_v|^{2\beta} f+ g .$$

Therefore applying Proposition \ref{prop-bouchut} of Bouchut (with the parameter $\alpha$ there replaced by $2\beta$), it follows that

$$||D_x|^{{2\beta}\over{1+2\beta}} f || \lesssim || |D_v|^{2\beta} f||^{1\over{1+2\beta}} . || - |D_v|^{2\beta} f +g ||^{{2\beta}\over{1+2\beta}}$$
$$\lesssim || |D_v|^{2\beta} f||^{1\over{1+2\beta}}  \lbrack  || |D_v|^{2\beta} f|| + |||D_v|^{2\beta} f||^{1\over{1+2\beta}} . || g||^{{2\beta}\over{1+2\beta}} \rbrack  .$$
Thus
\begin{equation}\label{eq-b}
|||D_x|^{{2\beta}\over{1+2\beta}} f || \lesssim || | D_v|^{2\beta} f|| + || |D_v|^{2\beta} f||^{1\over{1+2\beta}} . || g||^{{2\beta}\over{1+2\beta}} .
\end{equation}

{\bf Step 3:} Now apply $|D_x|^{\beta\over{1+2\beta}}$ on \eqref{eq2}, multiply by $|D_x|^{\beta\over{1+2\beta}} \bar f$ and integrate to get
\begin{equation}\label{eq-c}
\| |D_v|^\beta | D_x|^{\beta\over{1+2\beta}} f\| \lesssim \| |D_x|^{{2\beta}\over{1+2\beta}} f \|^{1\over 2} \| g\|^{1\over 2} .
\end{equation}

{\bf Step 4:} This step is different from Bouchut's arguments, in that we choose another mutiplier, taking into account the control for large frequency variable associated with $x$.

Considering \eqref{eq2}, multiply it by $ ( |D_v|^2 + |D_x|^{2\over{1+2\beta}} )^\beta \bar f$ (see the remark below for the choice of this multiplier) and integrate to get
$$ \int ( |D_v|^2 + |D_x|^{2\over{1+2\beta}} )^\beta \bar f .|D_v|^{2\beta} f = $$
$$- \re (( |D_v|^2 + |D_x|^{2\over{1+2\beta}} )^\beta \bar f, v.\nabla_x f )  + \re ( ( |D_v|^2 + |D_x|^{2\over{1+2\beta}} )^\beta \bar f ,g ) = I +II$$

Using Fourier transformation for example, it follows that
$$II \lesssim ( \| |D_v|^{2\beta} f\| + \| |D_x|^{{2\beta}\over{1+2\beta}} f\| ) \| g\| .$$

With the previous steps, we get
$$II \lesssim  \| |D_v|^{2\beta} f\| \| g\| + || |D_v|^{2\beta} f||^{1\over{1+2\beta}} . || g||^{1+ {{2\beta}\over{1+2\beta}}}
$$

On the other hand, using Parseval relation
$$I = -\re (( |D_v|^2 + |D_x|^{2\over{1+2\beta}} )^\beta \bar f, v.\nabla_x f) = -\re (( ( |\xi |^2 + | k |^{2\over{1+2\beta}} )^\beta \bar{\hat f} , k_j \partial_{\xi_j} \hat f$$
$$= \re (\partial_{\xi_j} [ ( |\xi |^2 + |k|^{2\over{1+2\beta}} )^\beta ] \bar{\hat f} , k\hat f) + \re ( ( |\xi |^2 + |k|^{2\over{1+2\beta}} )^\beta \partial_{\xi_j} \bar{\hat f} , k\hat f  ) .$$

Thus
$$I = \beta \re (\xi_j ( |\xi |^2 + |k|^{2\over{1+2\beta}} )^{\beta -1} \bar{\hat f} , k\hat f)  = \re \int \bar{\hat f} \hat f (\xi .k) ( |\xi |^2 + |k|^{2\over{1+2\beta}} )^{\beta -1}$$
$$\lesssim \int \bar{\hat f}\hat f |\xi | |k | |\xi |^{\beta -1} | k |^{{\beta -1 }\over{1+2\beta}} \lesssim \int |\xi |^\beta |k|^{\beta\over{1+2\beta}} \bar{\hat f} . | k|^{{2\beta}\over{1+2\beta}} \hat f$$
$$ \lesssim \| |D_v|^\beta |D_x|^{\beta\over{1+2\beta}} f \| . \| |D_x|^{{2\beta}\over{1+2\beta}} f \| .$$

Using the previous steps, it follows that
$$I \lesssim \| |D_v|^{2\beta} f\|^{3\over 2} \| g\|^{1\over 2} + \| |D_v|^{2\beta} f \|^{{3\over 2} {1\over{1+2\beta}}} \| g\|^{{3\over 2} {{2\beta}\over{1+2\beta}} +{1\over 2}} .$$

{\bf Step 5:} In conclusion, we get

$$I \lesssim \| |D_v|^{2\beta} f\|^{3\over 2} \| g\|^{1\over 2} + \| |D_v|^{2\beta} f \|^{{3\over 2} {1\over{1+2\beta}}} \| g\|^{{3\over 2} {{2\beta}\over{1+2\beta}} +{1\over 2}} $$
and
$$II \lesssim  \| |D_v|^{2\beta} f\| \| g\| + || |D_v|^{2\beta} f||^{1\over{1+2\beta}} . || g||^{1+ {{2\beta}\over{1+2\beta}}}
$$

Using Holder inequality, we get  for example, for small $\varepsilon >0$
$$   \| |D_v|^{2\beta} f\| \| g\| \lesssim \varepsilon^2 \| |D_v|^{2\beta} f\|^2 + C_\varepsilon \| g\| ^2 $$
$$ \| |D_v|^{2\beta} f\|^{3\over 2} \| g\|^{1\over 2} \lesssim \varepsilon^{4\over 3} \| |D_v|^{2\beta} f\|^2 +C_\varepsilon \| g\|^2$$
$$ || |D_v|^{2\beta} f||^{1\over{1+2\beta}} . || g||^{1+ {{2\beta}\over{1+2\beta}}} \lesssim \varepsilon^{2(1+2\beta )} || |D_v|^{2\beta} f||^2 + C_\varepsilon || g||^2$$

and therefore, it follows that
$$ \| |D_v|^{2\beta} f\| \lesssim C_\varepsilon \| g\| .$$
We get also from Step 2 that

$$\| |D_x|^{{2\beta}\over{1+2\beta}} f \| \lesssim \| g\| .$$

\begin{rema} 

\begin{enumerate}
\item To get also the same result for the derivative w.r.t. $t$, we repeat the above arguments, but with the multiplier $(|D_t|^{2\over{1+2\beta}} + |D_v|^2 + |D_x|^{2\over{1+2\beta}})^\beta$. As well we could have done the computations from the beginning with this multiplier.
\item We choose a mutiplier which is somehow singular near null value of the frequency variables. It would have been better to choose $(\delta + |D_v|^2 + |D_x|^{2\over{1+2\beta}})^\beta$, for a small $\delta$. Nothing is changed, except that now the upper bound involves $\| f\|$. In fact an even better choice would have been to choose $ (\delta +|D_v|^2 + <D_x>^{2\over{1+2\beta}} )^\beta$, whose symbol is smooth.
\item For the initial value problem, it might be interesting to consider the above multiplier by also $(\delta +|tD_v|^2 + <tD_x>^{2\over{1+2\beta}} )^\beta$.
\end{enumerate}
\end{rema}

\section{Proof of the second part of Theorem \ref{result}: the non constant coefficient case}

We now consider the model problem \eqref{eq1}.  As it should be clear now, the main issue is the estimation of the commutator of any smooth function with the operator $( <D_v>^2 + <D_x>^{2\over{1+2\beta}} )^\beta$, see the remarks in the previous section for the choice of this mutiplier.  
Recall that we assume
$$ a = b^2 \chi^2 + a_-$$
with $b \geq 0$ smooth and $\chi \geq 0$ compactly supported, and that we do not assume any lower bound on $b$.

Then, we write
$$ \int ( <D_v>^2 + <D_x>^{2\over{1+2\beta}} )^\beta \bar f .a.|D_v|^{2\beta} f $$
$$ = \int ( <D_v>^2 + <D_x>^{2\over{1+2\beta}} )^\beta \bar f .a_-.|D_v|^{2\beta} f + \int ( <D_v>^2 + <D_x>^{2\over{1+2\beta}} )^\beta \bar f .b^2\chi^2.|D_v|^{2\beta} f $$
$$ =I +II .$$

The first term $I$ is nice since it will give a lower bound as we wish. So we need to deal with $II$ to make appear a positive term and commutators terms:

We write below $P = ( <D_v>^2 + <D_x>^{2\over{1+2\beta}} )^\beta$ and $Q = |D_v|^{2\beta} $. Then

$$II =  \int ( |D_v|^2 + |D_x|^{2\over{1+2\beta}} )^\beta \bar f .b^2\chi^2.|D_v|^{2\beta} f  = \int b \chi P f . b\chi Q\bar f$$
$$ = \int \lbrace [b\chi ,P] f + P (b\chi f) \rbrace \lbrace [b\chi ,Q] f + Q (b\chi f) \rbrace$$
$$ = \int [b\chi ,P] f  .[b\chi ,Q] f + [b\chi ,P] f .Q (b\chi f)  + P (b\chi f)  .[b\chi ,Q] f  + P (b\chi f)Q (b\chi f)$$

The last term is positive. So we need to consider the first three terms, and in particular to study the commutator. 

The easiest commutator is $[b\chi , Q]$, Indeed, we note that

\begin{lemm} For $\beta \leq {1\over 2}$, one has
$$\| [b\chi , Q] f\| \lesssim c_b \| f\| $$
and for $\beta \geq {1\over 2}$, one has
$$\| [b\chi , Q] f\| \lesssim c_b [ \| |D_v|^{\beta -1/2}f\| + \| f\| ]$$
where the constant $c_b$ only depends on a finite number of derivatives of $b$.
\end{lemm}

\begin{proof}
This is a well know result, so we just sketch the main arguments. One possibility is to write (see for example Stein \cite{stein}), for some constant $c_n$
$$ Q f = c_n \int_h [f(v+h) -f(v) ] / | h|^{n+2\alpha} .$$
Then
$$b\chi Qf = c_n b\chi \int_h [f(v+h) -f(v) ] / | h|^{n+2\alpha} $$
$$ = c_n \int_h [b\chi f(v+h) -b\chi f(v) ] / | h|^{n+2\alpha} +\int_h [ b\chi (v) - b\chi (v+h) ] f(v+h)  / | h|^{n+2\alpha} .$$
Therefore:
$$[b\chi , Q] f = \int_h [b \chi (v) - b\chi (v+h) ] f(v+h)  / | h|^{n+2\alpha}$$
$$ = \int_z [ [ b\chi (v) - b\chi (z) ] f(z)  / | z-v|^{n+2\alpha} = \int_z K(v,z) f(z) dz$$

Note that:
$$| K(v,z) | \lesssim | z-v|^{n+2\alpha -1} \mbox{ and } |K(v,z) | \lesssim 1/| z-v|^{n+2\alpha} $$

Thus if $\beta < {1\over 2}$,  we can apply Shur's Lemma to see that $[\chi, Q]$ is a $L^2$ bounded operator. For larger values of $\beta$ we need to use the symetrized version of the integral expression of $Q$. In fact another method is the following: write $|D_v|^{2\alpha} =  [|D_v|^{2\alpha} - <D_v>^{2\alpha}] + <D_v>^{2\alpha}$. The first factor is clearly bounded in $L^2$ while the second one is dealt with the same method as the Lemma just below for the commutator with $P$.

\end{proof}

\begin{lemm} For $\beta \leq {1\over 2}$, one has
$$\| [b\chi ,P] f\| \lesssim c_b \| f\|$$
and for $\beta \geq {1\over 2}$, one has
$$\| [b\chi , P] f\| \lesssim c_b [ \| |D_v|^{\beta -1/2}f\| + \| f\|]$$
where the constant $c_b$ only depends on a finite number of derivatives of $b$.
\end{lemm}
\begin{proof}
 Set $P = p(D_x, D_v) = ( |D_v|^2 + |D_x|^{2\over{1+2\beta}} )^\beta$, with $p(k,\xi ) = ( <\xi >^2 + <k>^{2\over{1+2\beta}} )^\beta$. Let $\tilde b = b\chi$. Then $[P,\tilde b] u = P (\tilde b u) - \tilde b (Pu)$. Therefore
 $$\widehat{[P,\tilde b] u} \ (k,\xi ) = \widehat{P (\tilde b u) } - \widehat{\tilde b (Pu)}  = p(k,\xi ) \widehat{\tilde b} \ast \hat u (k,\xi ) - \widehat{\tilde b} \ast [ p \hat u]$$
$$ = \int_{k', \xi' } [ p(k,\xi ) - p(k',\xi ')]\widehat{\tilde b} (k - k', \xi - \xi ') \hat u (k',\xi ')$$

Set 
$$K (k,\xi ,k',\xi') = [ p(k,\xi ) - p(k',\xi ')]\widehat{\tilde b} (k - k', \xi - \xi ') = K_1 +K_2$$
with
$$K_1 = [ p(k,\xi ') - p(k',\xi ') ]\widehat{\tilde b} (k - k', \xi - \xi ')$$
and
$$ K_2 = [p(k, \xi ) - p(k , \xi ')]\widehat{\tilde b}(k - k', \xi - \xi ')$$

Ii is immediately seen that the second therm gives rise to a kernel for which we can apply Schur Lemma. That is we see that $|K_1| \lesssim | \widehat{\tilde b} (k - k', \xi - \xi ')| |k-k'|$, and then (for any small $\delta$)
$$\int_{k,\xi } <k> |\widehat{\tilde b}|(k,\xi ) \lesssim [\int <(k,\xi )>^{2n +2 +\delta} |\widehat{\tilde b} |^2 ]^{1\over 2}. $$
Thus, going back to the inverse Fourier transform, we have an operator $\tilde K_1$ such that $\|\tilde K_1 f\| \lesssim \| f\|$. For the part related to $K_2$, suppose first that $\beta \leq 1/2$. Then using Taylor's formulae, it follows that
$$\int_{k,\xi } |K_2 (k,\xi ,k',\xi ') | \lesssim \int_{k,\xi }  |\xi - \xi'||  \widehat{\tilde b}(k - k', \xi - \xi ')|  \lesssim[ \int_{k,\xi} <(k,\xi )>^{2n+ 2\delta}  |\widehat{\tilde b}|^2 ]^{1\over 2}$$ 
Schur Lemma applies and yields that it is a $L^2$ bounded operator. If $\beta \geq {1\over 2}$, we have a upper bound on $K_2$ as  $|\hat a||\xi -\xi '| [ <\xi>^{\beta -1/2} + <\xi '>^{\beta -1/2}$. Then, it's enough to use Petree's inequality to conclude.
\end{proof}

\begin{rema} The bound above depends on the norm of $b\chi$ is $H^{m}_{x,v}$ with $m = n+1+\delta$ for small $\delta >0$. It is likely not optimal. When $a$ does not depend on variable $x$, one can obtain better bounds.
\end{rema}

Putting together the two previous Lemma, we can conclude the proof of the second part of Theorem \ref{result}.


\begin{thebibliography}{99}

\bibitem{alex-solmax} R. Alexandre, Some solutions of the Boltzmann equation without angular cutoff. J. Statist. Phys. 104 (2001), no. 1-2, 327Ð358.

\bibitem{alex-review} R. Alexandre,
A review on Boltzmann equation with singular kernels. {\it Kinetic
and Related Models}, {\bf 2-4} (2009) 541-646.

\bibitem{al-1} R. Alexandre, L. Desvillettes, C. Villani and  B. Wennberg,
Entropy  dissipation and long-range interactions, {\it Arch.
Rational Mech. Anal.}, {\bf 152} (2000) 327-355.


\bibitem{amuxy-nonlinear-b} R.Alexandre, Y.Morimoto, S.Ukai, C.-J.Xu and T.Yang,
Uncertainty principle and kinetic equations, {\it J. Funct. Anal.},
{\bf 255} (2008) 2013-2066.


\bibitem{amuxy-nonlinear-3} R.Alexandre, Y.Morimoto,
S. Ukai, C.-J.Xu and T.Yang, Regularizing effect and local existence
for non-cutoff Boltzmann equation, {\it Arch. Rational Mech. Anal.},{\bf 198} (2010), 39-123.

\bibitem{amuxy3} R.Alexandre, Y.Morimoto, S. Ukai, C.-J.Xu
and T.Yang, Global existence and full regularity of the Boltzmann
equation without angular cutoff, Preprint HAL, http://hal.archives-ouvertes.fr/hal-00439227/fr/


\bibitem{al-3} R. Alexandre and  C. Villani, On the Boltzmann equation for
long-range interaction, {\it Comm. Pure  Appl.
Math.} {\bf 55} (2002) 30--70.

\bibitem{arsenio} D. Arsenio, Fron Boltzmann's equation to the incompressible Navier-Stokes-Fourier system with long-range interactions. Preprint.


\bibitem{bouchut} F. Bouchut, Hypoelliptic regularity in kinetic equations. J. Math. Pures Appl. (9) 81 (2002), no. 11, 1135Ð1159.

\bibitem{bouchut-desvillettes} F. Bouchut, L. Desvillettes, Averaging lemmas without time Fourier transform and application to discretized kinetic equations. Proc. Roy. Soc. Edinburgh Sect. A 129 (1999), no. 1, 19Ð36.


\bibitem{gressman-strain} P. Gressman, R. Strain, Global strong solutions of the Boltzmann equation without angular cut-off. arXiv:0912.0888.

\bibitem{herau-all} F. Herau, K. Pravda-Starov, Anisotropic hypoelliptic estimates for Landau-type operators, Preprint, http://arxiv.org/abs/1003.3265

\bibitem{hormander} L. Hormander, The analysis of linear Partial Differential Operators. Vol. I -- IV. Springer, Berlin 1990.

\bibitem{lerner-all} N. Lerner, Y. Morimoto, K. Pravda-Starov, Hypoelliptic Estimates for a Linear Model of the Boltzmann Equation without Angular Cutoff. Preprint, http://arxiv.org/abs/1012.4915

\bibitem{morimoto-all}Y. Morimoto and C.-J. Xu, Hypoelliticity for a class
of kinetic equations, {\it J. Math. Kyoto Univ.} {\bf 47} (2007),
129--152.

\bibitem{perthame} B. Perthame, Higher moments for kinetic equations: the Vlasov-Poisson and Fokker-Planck cases. Math. Methods Appl. Sci. 13 (1990), no. 5, 441Ð452.

\bibitem{stein} E. Stein, Singular Integrals and Differentiability Properties of Functions. Princeton University Press, New York, 1970.

\bibitem{villani}C. Villani, {\it A review of mathematical
topics in collisional kinetic theory}. Handbook of Fluid Mechanics.
Ed. S. Friedlander, D.Serre, 2002.

\bibitem{zaslavsky} G.-M. Zaslavsky, The physics of chaos in hamiltonian systems. World Scientific Publ. Co. , Imperial College Press 2007.


\end{thebibliography}
\end{document}